\documentclass[11pt]{amsart}

\theoremstyle{plain}
\newtheorem{theorem}{Theorem}[section]
\newtheorem{lemma}[theorem]{Lemma}
\newtheorem{proposition}[theorem]{Proposition} 
\newtheorem{corollary}[theorem]{Corollary}

\theoremstyle{remark}
\newtheorem{remark}[theorem]{Remark}

\theoremstyle{definition}
\newtheorem{definition}[theorem]{Definition} 
\newtheorem{notation}[theorem]{Notation} 

\newtheorem{example}[theorem]{Example}

\newcommand{\cD}{\mathcal{D}}

\newcommand{\cO}{\mathcal{O}}

\newcommand{\cE}{\mathcal{E}}

\newcommand{\C}{\mathbf{C}}

\newcommand{\Q}{\mathbf{Q}}
\newcommand{\Z}{\mathbf{Z}}
\newcommand{\Ps}{\mathbf{P}}
\newcommand{\bw}{\mathbf{w}}

\DeclareMathOperator{\sing}{sing}
\DeclareMathOperator{\rk}{rank}

\DeclareMathOperator{\rank}{rank}

\DeclareMathOperator{\codim}{codim}
\DeclareMathOperator{\prim}{prim}
\DeclareMathOperator{\Pic}{Pic}
\DeclareMathOperator{\NL}{NL}

\numberwithin{equation}{section}

\title[Complete intersection threefold with defect]{Nodal complete intersection threefold with defect}

\author[R.~Kloosterman]{Remke Kloosterman}
\address{Institut f\"ur Mathematik, Humboldt-Universit\"at zu Berlin,
Unter den Linden 6, D-10099 Berlin, Germany} 
\email{klooster@math.hu-berlin.de}
\date{\today}
\thanks{The author thanks Ivan Cheltsov, Slawomir Cynk, Vincenzo Di Gennaro, Brendan Hassett and Orsola Tommasi for several comments on a previous version of this paper. The author would like to thank the referee for several suggestions to improve the presentation.
The  author is partially supported by DFG-grant KL 2244/2-1.}
\begin{document}

\begin{abstract}
In this paper we show that a nodal complete intersection threefold $X$ in $\Ps^{3+c}$ with defect, but without induced defect, has at least $\sum_{i\leq j} (d_i-1)(d_j-1)$ nodes, provided either $c=2$ or $d_c>\sum_{i=1}^{c-1} d_i$ holds.
\end{abstract}

\maketitle

\section{Introduction}
 Let  $c$ be a positive integer.  For a nodal complete intersection  $X\subset \Ps^{3+c}$ of multidegree $d_1,\dots,d_c$ we define the defect of $X$ to be $h^{4}(X)-h^{2}(X)$. The defect is also the difference of the rank of the group of Weil divisors and the rank of Cartier divisors.
In this paper we aim to determine the minimal number of nodes such that $X$ has defect.

This minimal number has been determined by Cheltsov \cite{ChelPlane} in the case $c=1$. However, for the case $c>2$ only partial results are known. 
 Namely in the case $c=2$  Kosta \cite{Kosta} showed that if $X$ has defect, $d_1\leq d_2$ and the complete intersection is nondegenerate (see Section~\ref{secDeg}) then it has at least $(d_1+d_2-2)^2-(d_1+d_2-2)(d_1-1)$ nodes.

Cynk and Rams \cite{CynkRamsNonFac} considered complete intersection threefolds that are CR-nondegenerate in codimension 3 (see Section~\ref{secDeg}). 
They showed that  a nodal CR-nondegenerate complete intersection threefold with defect, such that the defect is caused by a smooth complete intersection surface, has at least $ \sum_{1\leq i\leq j\leq c} (d_i-1)(d_j-1)$ nodes, except if $c\leq 4$ and $d_1=\dots=d_c=2$ holds. In the latter  case the weaker bound  $\# X_{\sing}\geq 2^{c-1}$ holds. Moreover, they show that this bound is sharp, i.e., for each choice of $d_1,\dots,d_c$ they give an example of a complete intersection with either $ \sum_{1\leq i\leq j\leq c} (d_i-1)(d_j-1)$ or $2^{c-1}$ nodes and they conjecture that one can drop the two conditions on the surface causing the defect (smoothness and being a complete intersection).

In this paper we determine the minimal number of nodes for a nodal complete intersection to have defect, where we put a slightly different non-degeneracy condition than Cynk and Rams, and we assume that either $c=2$ or $d_1+\dots d_{c-1}<d_c$ holds.
Part of our strategy has been used in \cite{KloNod} to reprove Cheltsov's result for the case  $c=1$. In \cite{KloNod} we used results from commutative algebra also used in the proof of the explicit Noether--Lefschetz theorem by Green \cite{GreenF}.
It turns out that we can extend these techniques to the compete intersection case, but the arguments become much more technically involved and less elegant than in the hypersurface case. One of the difference is the fact, that there is a classical formula to calculate the defect of a hypersurface in terms of the defect of a linear system, but we were not able to find such a formula for the complete intersection case in the literature. (E.g., Kosta used a bound a for the defect, rather than a formula. See also Proposition~\ref{prpDefOld}.)
We prove a formula to compute the defect of a nodal complete intersection in terms of the defect of a certain linear system, see Proposition~\ref{prpDefCI}.

We use a different notions of nondegnerate complete intersections than Kosta and than Cynk--Rams.
We say that a complete intersection $X\subset \Ps^{3+c}$ of multidegree $(d_1,\dots,d_c)$, with $d_1\leq \dots\leq d_c$ has induced defect if there exists a four-dimensional complete intersection $Y\subset \Ps^{3+c}$ and a hypersurface $H\subset \Ps^{3+c}$ of degree $d_c$, such that $X=Y\cap H$ and such that for a general hyperplane $H'$ we have that $h^4(Y \cap H')>h^2(Y \cap H')$. If this is the case then the singular locus of $Y$ is one-dimensional. 
If $d_{c-1}<d_c$ holds then induced defect implies that $X$ is degenerate in codimension three in the sense of \cite{CynkRamsNonFac}.

Our main theorem is the following:
\begin{theorem} Let $X$ be a nodal complete intersection threefold in $\Ps^{3+c}$ of multidegree $(d_1,\dots,d_c)$ with defect, without induced defect. Suppose that either $c=2$ or $\sum_{i=1}^{c-1} d_i<d_c$ holds. Then $X$ has at least
\[ \sum_{i\leq j}(d_i-1)(d_j-1) \]
nodes.
\end{theorem}
(This is a combination of Theorem~\ref{thmCI} and Theorem~\ref{thmCIeq}.)

We will now briefly discuss the strategy of proof. As in \cite{KloNod} we study the Hilbert function of the ideal $I$ of the nodes of $X$, by studying the larger ideal $(I,\ell)$, where $\ell$ is a general linear form.
The main difference with the hypersurface case is that the smallest ideal $I'$ containing $I_H$ and such that $S/I'$ is  Gorenstein, is too big to obtain the desired lower bound for the nodes. Instead of  working directly with the complete intersection $X$ in $\Ps^{3+c}$ we work with an  associated hypersurface $Y$ in a  $\Ps^{c-1}$-bundle over $\Ps^{3+c}$. If $X$ is nodal then so is $Y$ and the nodes of $Y$ are in  one-to-one correspondence with those of $X$.  We analyze the ideal of the nodes of $Y$. The advantage of working with $Y$ is that one can rephrase the various nondegeneracy properties easily in terms of the ideal of the nodes of $Y$.
 
 The paper is organized as follows. In Section~\ref{sectMac} we recall several standard results on the Hilbert functions of ideals. In Section~\ref{secNodCI} we recall some standard results on the cohomology of nodal complete intersections. 
 In Section~\ref{secCayley} we discuss the above mentioned construction of $Y$ and give a formula to calculate the defect of a nodal complete intersection. In Section~\ref{secDeg} we discuss various notions of nondegenerate complete intersections and compare them. Finally in Section~\ref{secCIthm} we prove our main result.
 
\section{Macaulay's and Green's result}\label{sectMac}
We recall some results from commutative algebra. These results are also mentioned in \cite{KloNod} and we included them here for the reader's convenience.

Let $S=\C[x_0,\dots,x_n]$ and let $I\subset S$ be a homogeneous ideal. Let $h_I$ be the Hilbert function of $I$, i.e., $h_I(k)=\dim (S/I)_k$.

Let $d\geq 1$ be an integer.
Let $c:=h_I(d)$. We can write $c$ uniquely as
\[c= \sum_{i=1}^d \binom{i+\epsilon_i}{i}\]
with $\epsilon_d\geq \epsilon_{d-1}\geq \dots \geq \epsilon_1\geq -1$. We call this the \emph{(Macaulay) expansion} of $c$ in base $d$.
This expansion can be obtained inductively as follows: The number $\epsilon_d$ is the largest integer such that $\binom{d+\epsilon_d}{d}\leq c$. The numbers $\epsilon_i$ for $i<d$ are the coefficients in the expansion of $c-\binom{d+\epsilon_d}{d}$ in base $d-1$.

Using the Macaulay expansion of $c$ we define the following  numbers:
 \[c^{\langle d \rangle}:= \sum_{i=1}^d\binom{i+\epsilon_i+1}{i+1},\;  c_{\langle d \rangle}:=\sum_{i=1}^d \binom{i+\epsilon_i-1}{i} ,\; c_{*d}:=\sum_{i=2}^d \binom{i+\epsilon_i-1}{i-1}. \]
Note that  $c\mapsto c_{*d}$, $c\mapsto c^{\langle d \rangle}$ and  $c\mapsto c_{\langle d \rangle}$ are increasing functions in $c$.

Recall the following theorem by Macaulay:
\begin{theorem}[{Macaulay \cite{Mac}}]\label{thmMac} Let $V\subset S_d$ be a linear system and $c=\codim V$. Then the codimension of $V\otimes_{\C} S_1$ in $S_{d+1}$ is at most $c^{\langle d\rangle}$.
\end{theorem}

We apply this result mostly in the case where $V$ is the degree-$d$ part of an ideal $I$. In this case we can also obtain information on $h_I(d-1)$. 
\begin{corollary}\label{corMacDown}
Let $I\subset S$ be an ideal, $d\geq 2$ an integer and $c:=h_I(d)$. Then  
\[ h_I(d-1)\geq c_{* d} .\]
Moreover, if $\epsilon_{1}$ is nonnegative then $h_I(d-1)> c_{* d} $ holds.
\end{corollary}

For small $c$ we have the following Macaulay expansions in base $d$:
\begin{itemize}
\item For $c\leq d$ we have $\epsilon_d=\dots=\epsilon_{d-c+1}=0$ and $\epsilon_{d-c}=\dots=\epsilon_1=-1$. Hence $c^{\langle d\rangle}=c$.
\item For $d+1\leq c \leq 2d$ we have $\epsilon_d=1, \epsilon_{d-1}=\dots=\epsilon_{d-a}=0, \epsilon_{d-a-1}=\dots=\epsilon_1=-1$, where $a=c-d-1$.
Hence $c^{\langle d\rangle}=c+1$.
\item For $c=2d+1$ we have $\epsilon_d=\epsilon_{d-1}=1$ and all other $\epsilon_i$ equal $-1$. Hence $c^{\langle d \rangle}=2d+3=c+2$.
\end{itemize}

Applying the previous corollary repeatedly yields
\begin{corollary}\label{corMacLowDeg}
Let $I\subset S$ be an ideal, $d\geq 2$ an integer and $c:=h_I(d)$. For $0\leq k \leq d$ we have that
\[ h_{I}(k)\geq \left\{ \begin{array}{ll} \min(c,k+1) & \mbox{ if }c\leq d;\\
\min(k+(c-d),2k+1) &\mbox{ if } d+1\leq c\leq 2d;\\
2k+1 & \mbox{ if } c=2d+1.\end{array}\right.\]
\end{corollary} 

The following result will be used to detect the Hilbert polynomial of the ideal generated by $I_d$:
\begin{theorem}[{Gotzmann \cite{Gotz}}]\label{thmGotz}
Let $V\subset S_d$ be a linear system and let $J\subset S$ be the ideal generated by $V$. Set $c=h_J(d)$.  If $h_J(d+1)=c^{\langle d \rangle}$ then for all $k\geq d$ we have $h_J(k+1)=h_J(k)^{\langle k \rangle}$. In particular the Hilbert polynomial $p_J(t)$ of $J$ is given by
\[ \sum_{i=1}^d \binom{t+\epsilon_i}{t}\]
and the dimension of $V(J)$ equals $\epsilon_d$. 
\end{theorem}

We use this result mostly in the case where $c\leq d$:
\begin{corollary}\label{corGreen} Let $I\subset S$ be an ideal such that $h_I(d)\leq d$ and $I_{d+1}$ is base point free. Then for all $k\geq d$ we have $h_I(k+1)<h_I(k)$ or $h_I(k)=0$.
\end{corollary}
\begin{proof}
See \cite[Corollary 2.5]{KloNod}.
%
\end{proof}

%

\section{Nodal complete intersections}\label{secNodCI}
\begin{notation}
Let $n=2k+1$ be a positive odd integer,  $c$ be a positive integer, and $(w_0,\dots,w_{n+c})$ a sequence of positive integers. Let us denote with $\Ps:= \Ps(w_0,\dots, w_{n+c})$ the associated weighted projective space. Let $S=\C[x_0,\dots,x_{n+c}]$ be the graded polynomial ring such that $\deg x_i=w_i$.
\end{notation}

\begin{definition}
We say that a codimension $c$ complete intersection $X\subset \Ps$ is \emph{a nodal complete intersection of codimension $c$}, if
\begin{enumerate}
\item for all $p\in \Ps_{\sing}\cap X$ we have that $X$ is quasi-smooth at $p$ and
\item for all $p\in X_{\sing} \setminus(\Ps_{\sing}\cap X)$ we have that $(X,p)$ is an $A_1$-singularity.
\end{enumerate}
Let $\Sigma$ denote the set $X_{\sing} \setminus(\Ps_{\sing}\cap X)$.
\end{definition}

\begin{proposition}\label{prpLef} Let $X\subset \Ps$ be a nodal complete intersection of codimension $c$ then for $i<n$ 
\[ \dim H^i(X)=\dim H^i(\Ps).\]
Moreover, for $i<n-1$ we have
\[ \dim H^i(X)=\dim H^{2n-i}(X).\]
\end{proposition}

\begin{proof}
See \cite[Proposition 3.3]{KloNod}.
\end{proof}

The proof of the  above result suggests that   $h^{n+1}(X,\Q)$ may be strictly larger than $h^{n-1}(X,\Q)$.
\begin{definition}
The \emph{defect} $\delta$ of $X$ equals $h^{n+1}(X,\Q)-h^{n-1}(X,\Q)$.
\end{definition}

\begin{remark}
If $n=3$ then $\delta$ equals the rank of the group $\mathrm{CH}^1(X)/\Pic(X)$. Since this group is free, $\delta$ measures the failure of Weil divisors to be Cartier.
\end{remark}

\begin{lemma}\label{lemDefLoc} Let $X$ be a nodal complete intersection. Let $\cD$ be the equisingular deformation space of $X$. Then the locus 
\[ \{ X'\in \cD  \mid \delta(X')=\delta(X)\}\]
is a Zariski open subset of $\cD$.
\end{lemma}

\begin{proof}
See \cite[Lemma 3.6]{KloNod}.
\end{proof}

%
%
%

\section{Cayley trick}\label{secCayley}
Let $X\subset \Ps(w_0,\dots,w_{n+c})=:\Ps(\bw)$ be an $n$-dimensional complete intersection of multidegree $(d_1,\dots,d_c)$ with equations $f_1=\dots=f_c=0$. Set $\cE=\oplus_{i=1}^c \cO(d_i)$. Let $\Ps:=\Ps(\cE)$. Take coordinates $x_0,\dots,x_{n+c},y_1,\dots,y_c$ for $\Ps$ (cf. \cite[Section 2]{MavToric}). Let $Y\subset \Ps(\cE)$ be the hypersurface defined by
\[ F:=\sum y_if_i=0.\] 
Then $F\in H^0(\Ps(\cE), \cO_{\Ps(\cE)}(1))$.

The Cox ring of $\Ps(\cE)$ is generated by the $x_i$ and $y_j$. This is a bigraded ring with $\deg(x_i)=(w_i,0)$ and $\deg(y_j)=(-d_j,1)$. In particular, $\deg(F)=(0,1)$.
\begin{lemma}[Cayley trick]\label{lemCay} For $0\leq i \leq 2n$ we have a natural isomorphism
\[ H^i(X,\Q)\cong H^{2c+i-2}(Y,\Q)(c-1)\]
\end{lemma}
\begin{proof} Note that we have the following chain of isomorphisms
\begin{eqnarray*}
H^i_c(X)_{\prim}&\stackrel{Gysin}{\cong}&H^{i+1}_c(\Ps(\bw)\setminus X)\\
&\stackrel{PD}{\cong}&H^{2n+2c-i-1}(\Ps(\bw)\setminus X) (n+c-i-1)\\&=&H^{2n+2c-i-1}(\Ps\setminus Y)(n+c-i-1)\\&\stackrel{PD}{\cong}&H^{2c+i-1}_c(\Ps\setminus Y) (c-1) \\&\stackrel{Gysin}{\cong}&H^{2c+i-2}_c(Y)_{\prim}(c-1) \end{eqnarray*}
In the first and last line we used the Gysin exact sequence for cohomology with compact support. In the second and second to last line we used Poincar\'e duality (the complement of a hypersurface in weighted projective space is a $\Q$-homology manifold) and in the middle we used that $\Ps\setminus Y$  is a $\C^{c-1}$-bundle over $\Ps(\bw) \setminus X$.
\end{proof}

For a subvariety of a toric variety one has the notion of quasismoothness.   It follows easily that $X\subset \Ps(\bw)$ is not quasismooth at $p$ if and only if the rank of $M(p)=\left(\frac{\partial f_i}{\partial x_j}(p)\right)$ is strictly less than $c$. Similarly, we have that  $Y$ is not quasismooth at $(p,q)\in \Ps(\cE)$ if all the partial derivatives of $\sum_{i=1}^c y_if_i$ vanish simultaneously. 

\begin{definition}
Let $T$ be a toric variety and let $X\subset T$ be a subvariety, then with $X_{\sing}$ we denote the set of points $p\in X$ such that $X$ is not quasi-smooth at $p$.
\end{definition}

\begin{proposition} The projection $\Ps(\cE)\to \Ps(\bw)$ restricts to a surjective morphism $\psi:Y_{\sing}\to X_{\sing}$. For $p\in X_{\sing}$ the fiber $\psi^{-1}(p)$ is a linear space of dimension $c-\rk M(p)-1$. In particular, if $X$ is quasi-smooth then so is $Y$.
\end{proposition}

\begin{proof}
A point $(p,q)$ on  $Y$ is not quasismooth if and only if all the $\frac{\partial{F}}{y_i}$ and all the $\frac{\partial{F}}{x_j}$ vanish simultaneously at $(p,q)$.
 The partials with respect to $y_i$ are precisely the $f_i$. Hence $f_i(p)=0$ holds for $i=1,\dots c$, which yields   $\psi(Y_{\sing})\subset X$.

The partials of $F$ with respect to $x_i$ are $\sum_j y_j \frac{\partial f_j}{x_i}$. Hence, if $(p,q)$ is  a singular point of $Y$, then $\rank M(p)<c$ and therefore $\psi(Y_{\sing})\subset X_{\sing}$.

Let $p\in X_{\sing}$. Then $(p,q)$ is a singular point of $Y$ if and only if  $M(p)q^T=0$. In particular,
the fiber of $\psi$  is a linear space of  dimension $c-\rk M(p)-1$.
\end{proof}

\begin{proposition}\label{prpIsol} Suppose $p$ is an isolated hypersurface singularity of $X$. Then $ \psi^{-1}(p)$ is a point.\end{proposition}
\begin{proof}  
Locally $(X,p)$ is defined by $c$ equations $g_1=g_2=\dots=g_c=0$ in $\C^{n+c}$. Using the chain rule it follows that the rank the of the Jacobian matrix of $(g_1,\dots, g_c)$ is independent of the choice of local coordinates and equations. Since $(X,p)$ is a hypersurface singularity we can find local coordinates $x_1,\dots,x_{n+c}$ on $\C^{n+c}$ such that $g_i=x_i$ for $i=1,\dots,c-1$ and $g_c$ defines the hypersurface singularity, i.e., all its partials vanish at $p$. In particular, the rank of the Jacobian matrix is precisely $c-1$.

The previous proposition  implies that $\psi^{-1}(p)$ is a zero-dimensional linear space, i.e., a point.
\end{proof}

\begin{lemma}\label{lemNode} If $X$ is a nodal complete intersection, then $Y$ is a nodal hypersurface.\end{lemma}

\begin{proof}
Let $p$ be a singular point of $X$, and $(p,q)$ be the corresponding point on $Y$. Then we can find a coordinate change in a neighbourhood of $p$  such that $(X,p)$ is locally given by $\sum_{i=1}^{n+1}(x_i^2)=x_{n+2}=\dots=x_{n+c}=0$. Similarly, after a local coordinate change  in a neighbourhood of $(p,q)$ we have that   the equation for  $Y$ is given by
\[ \sum_{i=1}^{n+1}x_i^2y_1+\sum_{i=2}^{c} x_iy_i.\]
The corresponding singular point on $Y$ is $(0,\dots,0)\times (1:0:\dots:0)$. This point is clearly a node.
\end{proof}

For a finite subscheme $\Delta \subset \Ps(\bw)$ we say that the linear system of degree $k$ polynomials through $\Delta$ has defect  if for $I=I(\Delta)$ we have that $h_I(k)<p_I$ holds.

The following bound for the defect of complete intersections follows from the main result of Cynk \cite{CynkHS}: 
\begin{proposition}\label{prpDefOld} Let $X=V(f_1,\dots,f_c)$ be a three-dimensional nodal complete intersection in $\Ps^{3+c}$ such that $X_c:=V(f_1,\dots,\dots,f_{c-1})$ is smooth. Then the defect of $X$ is at most the defect of the linear system of degree $d_c+\sum_{i=1}^c d_i-\sum w_i$ polynomials through the points, where $X$ has a node.
\end{proposition}

Since the defect of $X$ and the defect of $Y$ are equal we can compute the defect on $Y$:
\begin{proposition}\label{prpDefCI} Let $X\subset \Ps(\bw)$ be a nodal complete intersection of odd dimension $n$ and let $Y\subset \Ps(\cE)$ be a corresponding hypersurface. Let us define $w:=\sum_{j=0}^{n+c} w_j$; $D:=\sum_{i=1}^c d_i$ and $m:=\frac{n+1}{2}$. Then the defect of $Y$ equals the defect of the linear system of polynomials of bidegree $(D-w,m-1)$ passing through the nodes of $Y$.
\end{proposition}

\begin{proof} The proof essentially follows the proof of \cite[Proposition 3.2]{DimBet} for hypersurfaces in weighted projective space. We refer to that paper for the details and give only a sketch of the proof:

The defect of $Y$ 
 equals the dimension of the cokernel of
\[ H^{n+2c-2}(Y\setminus Y_{\sing})\to H^{n+2c-1}_{Y_{\sing}}(Y). \]
Since $Y$ is a nodal hypersurface, the dimension of $H^{n+2c-1}_{Y_{\sing}}(Y)$ equals the number of nodes of $Y$ and its  Hodge structure is of pure $(m+c-1,m+c-1)$ type.
There is a residue map $H^{n+2c-1}(\Ps(\cE)\setminus Y)\to H^{n+2c-2}(Y\setminus Y_{\sing})$, which is a morphism of Hodge structures of degree $(-1,-1)$ and surjective for $n+2c-2\geq 3$.

Hence the defect of $Y$ equals the dimension of the cokernel
\[ F^{m+c} H^{n+2c-1}(\Ps(\cE)\setminus Y)\to  H^{n+2c-1}_{Y_{\sing}}(Y)=\oplus_{p \in Y_{\sing}} \C.\]

On $H^{n+2c-1}(\Ps(\cE)\setminus Y)$ we have the filtration by the pole order as defined in \cite{DelDim}. The main result of \cite{DelDim} shows that the filtration by the pole order is contained in the Hodge filtration. In particular, we have a surjective map $H^0(\cO(D-w,m-1))=H^0(K_{\Ps(\cE)}(m-1+c)Y)\to F^{m+c} H^{n+2c-1}(\Ps(\cE)\setminus Y)$.

The composed map $H^0(\cO(D-w,m-1))\to \oplus_{p\in Y_{\sing}}\C$ is the evaluation map, as in the case of hypersurfaces in $\Ps(\bw)$.
\end{proof}

\begin{example} In the case $n=3$ we have that the defect of $Y$ equals the cokernel of
\[ \oplus_{i=1}^c y_iS_{D+d_i} \to  \oplus_{p\in Y_{\sing}}\C.\]
\end{example}

\section{Degenerate complete intersections and induced defect}\label{secDeg}
There is a very simple construction of complete intersection with defect, which is often excluded if one attempts to determine the minimal number of nodes on a complete intersection with defect.
Let $c$ be an integer and  $n$ be an odd integer. Fix an integers $d_1\leq\dots \leq d_{c-1}$ and polynomials $f_i$ such that the $X_c=V(f_1,\dots,f_{c-1})$ is a complete intersection fourfold in $\Ps^{3+c}$ satisfying $h^{n+3}(X_c)\neq h^{n-1}(X_c)=1$. 
Then $X_c$ has a one-dimensional singular locus, say of degree $s$. Let $\gamma \in H^{n+3}(X_c,\Z)$ be such that $\gamma_{\prim}\neq 0$. For a general polynomial $g$ of degree $d_c$ the complete intersection $X(g):=X_c\cap V(g)$  has defect and this defect is caused by the cycle $\gamma\cap V(g)\in H^{n+1}(X(g),\Z)$. The number of singular points of $X(g)$ equals $sd_c$. In particular, the number of singular points grows linearly in $d_c$. With some effort one can produce examples of this type where the transversal types of all non-isolated singularities of $X_c$ are $A_1$, and therefore $X(g)$ is a nodal hypersurface.

The lower bound on the number of nodes we prove later on is quadratic in each of the $d_i$ and hence we have to exclude this example.  In the literature there are various notions of nondegenerate complete intersections. Each of these notions are attempts to exclude examples as above.
\begin{definition} Let $n$ be an odd integer and $d_1,\dots,d_c$ integers such that $d_1\leq d_2\leq\dots \leq d_c$. Let $X\subset \Ps^{n+c}$ be an $n$-dimensional complete intersection of multidegree $(d_1,\dots,d_c)$ with isolated singularities. 

We call $X$ \emph{nondegenerate} if for every choice of $(f_1,\dots,f_c) \in S_{d_1}\oplus \dots \oplus S_{d_c}$ such that $X=V(f_1,\dots,f_c)$ and for every $j\in \{1,\dots,c-1\}$ the complete intersection $V(f_1,\dots,f_j)$ is smooth.

We call $X$ \emph{nondegenerate in dimension $d$}, respectively, \emph{in codimension $e$}, if for every choice of $(f_1,\dots,f_c) \in S_{d_1}\oplus \dots \oplus S_{d_c}$ such that $X=V(f_1,\dots,f_c)$ and for every $j\in \{1,\dots,c-1\}$ the singular locus of the  complete intersection $V(f_1,\dots,f_j)$ has dimension at most $d-1$, respectively  codimension at least $e+1$.

We call $X$ \emph{without induced defect}  if for every choice of $(f_1,\dots,f_c) \in S_{d_1}\oplus \dots \oplus S_{d_c}$ such that $X=V(f_1,\dots,f_c)$, for every $j\in \{1,\dots,c-1\}$ and for every general hyperplane $H$ we have $h^{n-1+2(c-j)}(X_j \cap H)=h^{n-1}(X_j \cap H)$.

We call $X$ \emph{CR-nondegenerate} if for a choice of $(f_1,\dots,f_c) \in S_{d_1}\oplus \dots \oplus S_{d_c}$ such that $X=V(f_1,\dots,f_c)$ and for every $j\in \{1,\dots,c-1\}$ the complete intersection $V(f_1,\dots,f_j)$ is smooth.

We call $X$ \emph{CR-nondegenerate in dimension $d$}, respectively, \emph{in codimension $e$}, if for a choice of $(f_1,\dots,f_c) \in S_{d_1}\oplus \dots \oplus S_{d_c}$ such that $X=V(f_1,\dots,f_c)$ and for every $j\in \{1,\dots,d_c-1\}$ the singular locus of the complete intersection $V(f_1,\dots,f_j)$ has dimension at most $d-1$, respectively  codimension at least $e+1$.
\end{definition} 

\begin{remark} A more natural  definition of \emph{without induced defect} would be to require  $h^{n+1+2(c-j)}(X_j )=h^{n-1}(X_j)$ for every choice of $(f_1,\dots,f_c) \in S_{d_1}\oplus \dots \oplus S_{d_c}$ such that $X=V(f_1,\dots,f_c)$ and for every $j\in \{1,\dots,c-1\}$  we have. However this definition seems insufficient for the proofs in the next section.
\end{remark}

\begin{remark} 
If $X$ has induced defect then there is a choice of $f_1,\dots,f_c$ such that $X=V(f_1,\dots,f_c)$ and a $j<c$ such that $V(f_1,\dots,f_j)$  is singular in dimension $c-j$ (or codimension $n$).
Hence if $X$ is nondegenerate in codimension $n$ then $X$ is without induced defect.

If $d_{c-1}\neq d_c$ holds then $X$ is CR-nondegenerate in codimension $n$ if and only if it is nondegenerate in codimension $n$. However, if $d_c=d_{c-1}$ the  both notions  may  differ:  In \cite{CynkRamsNonFac} the authors study complete intersection threefolds and conjecture that if a nodal complete intersection threefold is CR-nondegenerate in codimension three then either $X$ has at least $\sum_{i\leq j} (d_i-1)(d_j-1)$ nodes or one has $c\leq 4$ and $d_1=\dots=d_c=2$. The examples of the last kind they provide have induced defect.
\end{remark}

\begin{example}\label{example22}
Suppose $n=3$, $c=2$ and $d_1=d_2=2$. Then by Proposition~\ref{prpDefCI} the defect of $X$ equals the dimension of the cokernel
\[ \C y_1\oplus \C y_2\to \oplus_{p\in Y_{\sing} }\C.\]
Hence $X$ has positive defect if and only if $Y$ has  at least three nodes, or $Y$ has two nodes with the same  $(y_1:y_2)$-coordinate.

Suppose we are in the latter case. Then after choosing different generators $(f_1,f_2)$ for $I(X)$  we may assume $y_2=0$. This implies that $V(f_1)$ is singular at the two nodes of $Y$. Since $f_1$ is a quadric it follows that the singular locus of $f_1$ is  one-dimensional.  In particular we may assume that $f_1=\sum_{i=0}^3 x_i^2$.  Now $V(f_1)$ contains the linear $3$-space $x_0=ix_1, x_2=ix_3$. Hence if $X$ has defect and has two nodes then the  defect is induced. Cynk and Rams showed that there exist CR-nondegenerate complete intersections of degree $(2,2)$ with two nodes that have defect. 
\end{example}

\section{Complete intersection threefolds}\label{secCIthm}
Let $X=V(f_1,\dots,f_c)\subset \Ps^{3+c}$ be a nodal complete intersection of multidegree $(d_1,\dots,d_c)$. Throughout this section we assume that $d_1\leq\dots\leq d_c$. Moreover, if  $c=2$ and $d_1=d_2=2$ holds then it follows from Example~\ref{example22} that a complete intersection with defect, but without induced defect has at least 3 nodes. Hence we proved  Theorem~\ref{thmCIeq} in this case. 
Throughout this section  we will assume that either $c>2$ or $d_2>2$ holds. We need this assumption since Proposition~\ref{prpW} does not hold true if a general hyperplane section of $X$ is a rational surface.

Let $Y\subset \Ps(\cE)$ be the hypersurface constructed by the Cayley trick and let $I\subset S(\Ps(\cE))$ be the ideal of the nodes of $Y$.
Set now $D:=\sum d_i$ and consider  $I_{(D-4-c,1)}$. Since all partials of $F$ are contained in $I$ we have for each $j$ that $\sum_i f_{i,x_j}y_i$ belongs $I_{(D-4-c,1)}$. Moreover, if $d_i<d_j$ then $f_iy_j$ belongs also to $I_{(D-4-c,1)}$.

Construct now the following $R:=\C[x_0,\dots,x_{3+c}]$-module $W$, containing $I_{(D-4-c,1)}$:
\[ W_k:=\left\{ (h_1,\dots,h_c)\in \oplus_j R_{k-d_c+d_j} \left|\begin{array}{l} \sum_{j=1}^c h_j(x_0,\dots,x_{3+c})y_j =0 \mbox{ for}\\\mbox{all } (x_0,\dots,x_{n+c};y_1,\dots,y_c)\in Y_{\sing} \end{array}\right.\right \}.
\]
Note that the codimension of $W_k$ in  $\oplus_i R_{k-d_c+d_i}$ is at most the number of  nodes on $Y$. From Proposition~\ref{prpIsol} and Lemma~\ref{lemNode} it follows that $X$ and $Y$ have the same number of nodes.

Let $H\subset \Ps^{3+c}$ a general hyperplane. Then $X_H$ is a smooth complete intersection surface in $\Ps^{2+c}$. Denote $g_i:=f_i|_H$.
Without loss of generality we may assume that $H$ is given by $x_{3+c}=0$. With $W'$ we denote the $S:=\C[x_0,\dots,x_{2+c}]$-module $W|_{x_{3+c}=0}$.
Then for all $k$ we have
\[\# X_{\sing}\geq h_{W}(k)\geq\sum_{j=0}^k h_{W'}(j)\]

In the sequel we provide a lower bound for $h_{W'}(k)$.  We provide first a lower bound for $h_{W'}(k)$ for $k\geq d_c$ and here we exploit that $X$ has no induced defect.

To obtain a good bound for lower degrees turns out to be more difficult: 
In the cases we considered before we could use Gorenstein duality, but this does not seem to work in the complete intersection case.  We provide a lower bound for $h_W(k)$ for $k<d_c-1$ in the case that $c=2$, or if $d_c$ is large compared with $\sum_{j=1}^{c-1} d_j$.

Note that the module $W$ depends on the choice of generators for $I(X)$, but that $h_W$ does not depend on it. For a subspace $V\subset \oplus S_{k+d_i-d_c}$ and a point $p\in \Ps^{2+c}$ we denote with $V(p)\subset \C^c$ the vector space obtained by evaluating all elements of $V$ at $p$.

\begin{proposition}\label{prpW} Let $X$ be a nodal complete intersection in $\Ps^{3+c}$ of multidegree $(d_1,\dots, d_c)$ with defect, but without induced defect. Assume $c\geq 3$ or $c=2$ and $d_2>2$. Then there exists a codimension one subspace $V_{D+d_c-c-3}$ of $\oplus_i S_{D+d_i-c-3}$ such that
\begin{enumerate}
 \item $W'_{D+d_c-c-3}\subset V_{D+d_c-c-3}$.
 \item For every choice of generators for $I(X)$ we have \\
$ 
  0\oplus \dots \oplus 0\oplus  S_{D+d_c-c-3} \not \subset V_{D+d_c-c-3}.
$ 
\item For $k\leq D+d_c-c-3$ let  $V_{k}\subset \oplus S_{k-d_c+d_i}$ be the largest subspace such that $V_{k} S_{D+d_c-c-3-k}$ is contained in $ V_{D+d_c-c-3}$. Then for all $p\in \Ps^{2+c}$ we have that $ V_{d_c}(p)=\C^c$. 
\end{enumerate}
\end{proposition}

\begin{proof}
We have $h_{W}(k)=\#Y_{\sing}$ for $k$ sufficiently large. Since $X$ has defect, we have that $h_W(D+d_c-c-4)<\#Y_{\sing}$. Therefore there is some $k\geq D+d_c-c-3$ such that $h_{W'}(k)>0$ and hence $h_{W'}(k)>0$ for all $d_c-d_1\leq k \leq D+d_c-c-3$. 
In particular, we can find a subspace $V_{D+d_c-c-3}$ satisfying the first condition.

Note that by construction $X$ has a Weil divisor $P$ that is not $\Q$-Cartier. Hence $X_H$ contains a divisor which is not the multiple of a hyperplane section. The elements of $W$ are tangent vectors to the equisingular deformation space of $X$ and these equisingular deformations preserve the defect of $X$ (Lemma~\ref{lemDefLoc}). Hence $W'_{d_c}$ is  contained in the tangent space to some component $L$ of the Noether-Lefschetz locus of complete intersections of multidegree $(d_1,\dots,d_c)$ in $\Ps^{2+c}$.

Consider now $X'_c=V(f_1,\dots,f_{c-1})\cap H$, which is either smooth or has isolated singularities. Since $H$ is a general hyperplane and the transversal type of the non-isolated singularities of $V(f_1,\dots,f_{c-1})$ is $A_1$ it follows that $X'_c$ is a nodal threefold.  

Define the Noether-Lefschetz locus inside $|\cO_{X'_c}(d_c)|$ as the locus of smooth surfaces with Picard number at least one.
 If the Noether-Lefschetz locus is Zariski open in $|\cO_{X'_c}(d_c)|$ then it follows from  \cite[Theorem 2.1(c)]{FraDiGNL} that $h^4(X'_c)\geq 2$. Since $X$ does not have induced defect we can  exclude this. Since the closure of each irreducible component of the Noether-Lefschetz locus is a proper subset  of $|\cO_{X'_c}(d_c)|$ it follows from the same result that 
\[W'_{D+d_c-c-3}\cap (0\oplus \dots\oplus0\oplus S_{D+d_c-c-3})\neq (0\oplus \dots \oplus 0\oplus S_{D+d_c-c-3}).\] Hence we can choose $V_{D+d_c-c-3}$ satisfying both (1) and  (2).

Denote with $e_i\in \C^c$ the $i$-th standard basis vector. By construction  the partials $((g_{1})_{x_i},\dots, (g_{c})_{x_i})$ and the elements  $g_{j}e_i$ are contained in $W$. If $p\in \Ps^{2+c}$ is a point where the Jacobian matrix of $(g_{1},\dots,g_{c})$ has full rank, then $\dim W'_{d_c-1}(p)=c$ and hence $\dim V_{d_c-1}(p)=\dim V_{d_c}(p)=c$. Hence we need only to consider points $p$ such that  the Jacobian matrix is not of full rank. Since $X'$ is smooth  there is some $i$ such that $f_i(p)\neq 0$. 
Then for any $j$ with $d_j\geq d_i$ we have $e_j\in W'_{d_c}(p)$. Hence $p$ is a singular point of a partial complete intersection $V(f_1,\dots,f_k)$ with $d_k<d_i$. 

Let $p$ be a point such that $\dim V_{d_c}(p)<c$. Then there exists $a_1,\dots,a_c\in \C$ such that $\sum a_ih_i(p)=0$ for all $(h_1,\dots,h_c)\in W_{d_c}$.  
From $e_c\in V_{d_c}(p)$ it follows that $a_c$ vanishes.

Note that the same linear relation holds for $(W'_{d_c}\cdot  S_{D-c-2})(p)$. From \cite{DelDim} it follows that there is a surjection $\oplus_{i=1}^c S_{D+d_i-c-2}\to H^{1,1}(X_H,\C)_{\prim}$. Let $\gamma\in H^2(X_H,\C)_{\prim}$ be a nonzero class mapped to the subspace of $W'_{D+d_c-c-2}$ where $\sum a_ih_i(p)=0$ holds. Let $\Lambda\subset H^2(X_H,\Z)$ be   sub-Hodge structure which contains $\gamma$. Then for any infinitesimal deformation of $X_H$ in $|\cO_{X'_c}(d_c)|$ we have that $\Lambda$ remains a subhodge structure of $H^2$. This is ruled out by \cite[Theorem 2.1(c)]{FraDiGNL}.
\end{proof}

\begin{remark}
If we assume that $X$ is nondegenerate in codimension three then $X'_c$ (as constructed in the above proof) is a smooth threefold. In this case we have that the monodromy representation on $H^2(X_H,\C)_{\prim}$ is irreducible. If $c>2$ or $d_2>2$ this implies that the Noether-Lefschetz locus in $|\cO_{X'_c}(d_c)|$ is not a Zariski open subset. From this we get points (2) and (3) almost directly. Hence in this case we can avoid the application of \cite{FraDiGNL}. A similar reasoning can also avoid the other application of this paper in the case that $X$ is nondegenerate in codimension three.

If $X$ has induced defect then it is easy to see that the conclusion of the above Proposition may not hold.
\end{remark}

For the rest of this section we will choose a $V_{d_c+D-c-3}$ satisfying the conclusion of the proposition and we set 
\[ V_k:=\{(h_1,\dots,h_c)\mid S_{D+d_c-c-3-k} (h_1,\dots,h_c)\in V_{D+d_c-c-3}\}\]
for $k\leq D+d_c-c-3$ and 
\[ V_k=\oplus S_{k+d_i-d_c}\]
for $k\geq D+d_c-c-2$.

Then for all integers $k$ we have
\[ \#X_{\sing}\geq \sum_{j=0}^k h_V(k).\]
Note that by the construction of $V$ we have that the rows of the Jacobian matrix of the complete intersection $V(g_1,\dots,g_c)$ are contained in $V$ and that $g_ie_j\in V$.

\begin{corollary}\label{corHighDeg} Let $X=V(f_1,\dots,f_c)$ be a complete intersection with defect, but without induced defect. Then for $d_c\leq k \leq D+d_c-c-2$ we have
\[ h_V(k)\geq D+d_c-c-2-k.\]
\end{corollary}

\begin{proof}
Since $X$ has no induced defect  we have for all $p\in \Ps^{2+c}$  that the dimension of $V_{d_c}(p)$ equals $c$. It follows from \cite[Proposition 1]{KimNL} that for $k\geq d_c$ we have either $h_V(k)>h_V(k+1)$ or $h_V(m)=0$ for all $m\geq k$. From $h_V(d_c-c-3+\sum d_j)\geq 1$ it follows that
$h_V(k)\geq -k+d_c-c-2+\sum_j d_j$ for $d_c\leq k\leq d_c-n-2-\sum d_j$.
\end{proof}

We want to bound $h_V$ in degree at most $d_c-1$. In order to do this we will define a filtration on $V$.
\begin{notation}
Let $F^iV\subset \oplus_{j\geq i} S(d_j-d_c)$ be the natural projection of $V\cap (0\oplus \dots 0\oplus S(d_i-d_c)\oplus \dots S(d_c-d_c))$ on the last $c-i+1$ factors.

Let us define $P^iV_j$ by $\mathrm{pr}_1(F^i V)(d_c-d_j)$, i.e., the projection of $F^i$ onto the first factor, with an appropriate degree shift.
\end{notation}

It is easy to show that
\[ h_{F^iV}(k)=h_{F^{i+1}V}(k)+h_{P^iV}(k+d_i-d_c)\]
for all $k$ and  all $i\in \{1,\dots,c-1\}$.
Note that $P^iV$ for $i=1,\dots,c$ and $F^cV$ are ideals and that by construction $S/F^cV$ is Artinian Gorenstein of socle degree $D+d_c-c-3$.

\begin{lemma}\label{lemHighFc}
Let $X=V(f_1,\dots,f_c)$ be a complete intersection with defect, but without induced defect. Then for $D-c-1\leq k \leq D+d_c-c-3$ we have
\[ h_{F^cV}(k)\geq D+d_c-2-c-k.\]
\end{lemma}

\begin{proof}
Note that the rows of the Jacobian matrix of $X_H$ are contained in $V_{d_c-1}$. From this it follows that the determinant of the Jacobian matrix is contained in $F^cV_{D-c}$. Recall that $g_i=f_i|_H$ and that $g_i\in F^cV$ for all $i$.
Hence a base point of $F^cV_{D-c}$ is a singular point of $X_H$. Since $X_H$ is smooth it follows that $F^cV_{D-c}$ is base point free. Using that $h_{F^cV}(D+d_c-c-3)\geq 1$ holds and Corollary~\ref{corGreen} we obtain that $h_{F^cV}(k)=D+d_c-c-2-k$ holds for $D-c-1\leq k \leq D+d_c-c-2$.
\end{proof}

\begin{lemma}\label{lemLowFc} For $k\leq d_c-2$ we have $h_{F^cV}(k)\geq k+1$.\end{lemma}
\begin{proof}
This follows from Gorenstein duality 
\[ h_{F^cV}(k)=h_{F^cV}(D+d_c-c-3-k)\]
and 
 the previous lemma.
\end{proof}
Recall that we set $D=\sum_{i=1}^c d_i$. In the sequel we will concentrate on the case $D<2d_c$. If $c=2$ this covers every case except $d_1=d_2$, which we will treat separately. However, for $c>2$ the condition $D<2d_c$ requires  $d_c$  to be  large compared with $d_1,\dots,d_{c-1}$.

The following result turns out to be very useful.

\begin{lemma}\label{lemFcsmall} Suppose  that $D\leq 2d_c$ and that at least one of the following holds:
\begin{itemize}
\item $h_{F^cV}(1)=2$
\item  $d_c\geq 3$, $h_{F^cV}(d_c-1)=h_{F^cV}(d_c-2)-1$ and $h_{F^cV}(d_c-2)\leq 2d_c-4$
\item $d_c\geq 4$, $h_{F^cV}(d_c-2)=h_{F^cV}(d_c-3)-1$ and  $h_{F^cV}(d_c-3)\leq 2d_c-6$.
\end{itemize}
 Then $X_H$ contains a line.
\end{lemma}

\begin{proof}
In the first case there is a line $\ell$, such that $I(\ell)\subset F^cV$. In the second case we can apply Theorem~\ref{thmGotz} and we obtain that the base locus of $F^cV_{d_c-1}$, resp. $F^cV_{d_c-2}$ consists of a line $\ell$ together with finitely many points.

Let $X'$ be the partial complete intersection $V(g_1,\dots,g_{c-1})$. The space $F^cV_d$ is contained in the tangent space to a component of the Noether-Lefschetz locus in $|\cO_{X'}(d_c)|$. We want to apply the strategy from \cite[Section 7]{OtwNLgen} to conclude that $\ell\subset X_H$. However, in that paper it is assumed that $X'$ is smooth and $d_c$ is sufficiently large. In the case where $X'$ is smooth, i.e., $X$ is nondegenerate in codimension 3, the proof needs little adaptation:

 We mostly use the results of Section 1 and 3 of \cite{OtwNLgen}, which hold for arbitrary $d_c$. To conclude that $X_H$ contains a line we need also the results of Section 7 of \emph{loc. cit.} and this is the place where Otwinowska needs that $d_c$ is sufficiently large. In this section one chooses a general  linear space $L\subset \Ps^{2+c}$ of codimension three and uses it to define a space $\tilde{T}(L)\subset \oplus H^0(X',\mathcal{T}_{X'}(i))$. In particular, for any $v\in \tilde{T}(L)$  and any element $f$ of the coordinate ring of $X'$ one can consider the  Lie derivative $L_v f$.

If $\ell'$ is a line not contained in $X_H$, then it is shown in the proof of \cite[Lemma 10]{OtwNLgen} that there is a $v \in \tilde{T}(L)$ such that for any point $p\in X_H\cap \ell'$ we have  $L_vg_c(p)\neq 0$. In particular, the ideal $I'$ generated by $g_1,\dots,g_{c-1},g_c$, $I(\ell')$ and $L_vg_c$ defines the empty scheme. We can find a $v$ such that $\deg L_v(g_c)=D-c$, (i.e., we could choose $v$ such that $L_v(g_c)$ is a maximal minor of the Jacobian matrix of $X_H$) and using Corollary~\ref{corMacDown} we obtain that 
\[ h_{I'}(k)\leq \left\{ \begin{array}{ll} k+1 & \mbox{if } 0\leq k<d_c\\ 
h_{I'}(k)\leq d_c&\mbox{if } d_c\leq k \leq D-c\\         
h_{I'}(k)\leq D-c+d_c-1-k& \mbox{if } D-c\leq k \leq D-c+d_c-1\\
                         \end{array}\right.\]
 In particular, $h_{I' }(D-c+d_c-1)=0$ holds.

 Let $E_1:=(F^cV:g_c)$ and let $R_L$ be a linear form vanishing on the ``cone'' $C(L,\ell)\cong \Ps^c$. From \cite[Lemme 11]{OtwNLgen} it follows that $(I(\ell),g_1,\dots,g_c,L_vg_c \mid v\in \tilde{T}(L))$ is contained in $(E_1:R_L)$ and from \cite[Lemme 12]{OtwNLgen} it follows that $h_{(E_1:R_L)}(3d_c-2)\neq 0$. 
Since $D\leq 2d_c$ we have $D-c+d_c-1\leq 3d_c-2$ and we obtain that $\ell \subset X_H$. 

If $X'$ has isolated singularities we have to proceed in a slightly different way. Since $X'$ is smooth in a neighbourhood of $X_H$, we can consider $X_H$  as a hypersurface in a resolution of singularities $\tilde{X'}$ of $X'$. This allows us to define the ideals $E_i$ as in \cite[Section 1]{OtwNLgen}. More problematic is the definition of $\tilde{T}$ and $\tilde{T}(L)$. Let $C_{X'}^*$ be the affine cone over $X'$ minus the vertex and $C^*_{\tilde{X'}}$ the fiber product $C^*_{X'} \times_{X'} \tilde{X'}$. We can now take the global section of the tangent sheaf of $C^*_{\tilde{X'}}$ and define a grading as in \cite[Section 3]{OtwNLgen}. Then Proposition 1 of \emph{loc. cit.} holds true.

We define now $\tilde{T}(L)$ similarly as in \cite[Section 7]{OtwNLgen} but we have to restrict to elements in $\tilde{T}$ such that their pushforward to $X'$ is well defined. If $\ell'\not\subset X'$ then we can find again a $v\in T(L)$ such that $h_{(I(\ell'),f_1,\dots,f_c,L_vf_c)}(D-c+d_c-1)=0$. Similarly we can show that  $h_{(E_1:R_L)})(3d_c-2)\neq 0$ and we obtain again that $\ell \subset X_H$.
\end{proof}

\begin{lemma} \label{lemLine} Suppose that $X$ has defect, but no induced defect. Moreover, suppose that  the intersection of a general member of the equisingular deformation space of $X$ with a general hyperplane is a surface containing a line. Then the number of nodes of $X$ is at least \[ \sum_{i\leq j}(d_i-1)(d_j-1).\]
\end{lemma}

\begin{proof}
The map $(X_t,H)\mapsto X_{t,H}$ defines a map from the product of the equisingular deformation space of $X$ and $(\Ps^{3+c})^{\vee}$ to a component $\NL(\gamma)$ of the Noether-Lefschetz locus of complete intersections of degree $(d_1,\dots,d_c)$ in $\Ps^{2+c}$.
Since $(X,H)$ are chosen general and $X_H$ contains a line we may assume that the image of this map lands in the component parametrizing surfaces containing a line, i.e, we can take  $\gamma$ to be the primitive part of the class of the line. Then $W_{D+d_c-c-3}$ is contained in the lift to $\oplus S_{D+d_i-c-3}$ of the orthogonal complement of $\gamma$ in $H^{1,1}(X,\C)_{\prim}$ and we may take $V_{D+d_c-c-3}$ to be precisely this orthogonal complement.

It follows directly from Noether-Lefschetz theory that $\oplus I(\ell)_{D+d_i-c-3}\subset V_{D+d_i-c-3}$ and therefore that $L:=\oplus I(\ell)(d_i-d_c)$ is contained in $ V$. If $p\not \in \ell$ then $\dim L_{d_c-1}(p)=c$. If $p\in \ell$ then $p\in X_H$ since $X_H$ is smooth at $p$ and the rows of the Jacobian matrix of $X_H$ are contained in $V_{d_c-1}$ it follows that  $\dim V_{d_c-1}(p)=c$.

In particular $\dim V_{d_c-1}(p)=c$ for all $p\in \Ps^{2+c}$. From \cite[Proposition 1]{KimNL} it follows now that $h_V(d_c-1)\geq D-c-1$.  
Let $p$ be a point on $\ell$.  Then there are elements $v_1,\dots,v_c\in V$, each of degree at most $d_c-1$, such that the determinant of $v_1,\dots,v_c$ does not vanish at $p$. This yields $c$ generators of $V$ which are independent modulo $L$. However the determinant of $v_1,\dots,v_c$ vanishes at some point on $\ell$, hence there are at least $c+1$ generators of $V$ of degree at most $d_c-1$, which are independent modulo $L$. Since $h_V(d_c-1)-h_L(d_c-1)\leq c+1$ holds, we obtain that each of these generators in $c+1$ and that $L_k=V_k$ for $k\leq d_c-2$. In particular, 
\[ h_V(k)=\sum_{i=1}^{c} \max(0,k+1-d_c+d_i).\]
Combining this with $h_V(k)\geq D+d_c-c-2-k$ for $d_c-1\leq k \leq D+d_c-c-3$ yields
\[ \sum_{k=0}^{D+d_c-c-3} h_V(k)\geq \sum_{i\leq j}(d_i-1)(d_j-1).\]
\end{proof}

We denote with $h_{V_L}$ the Hilbert function of the  $S$-algebra constructed in the previous proof. Then
\[ h_{V_L}(k)=\left\{ \begin{matrix} \sum_{i=1}^c \max(0,k+1-d_c+d_i)  & k\leq d_c-2 \\ D+d_c-c-2-k & d_c-1\leq k \leq D+d_c-c-2 \\ 0 & k\geq D+d_c-c-2 \\ \end{matrix}\right.\]

We will next focus on the case where $X_H$ does not contain a line.
We consider first the case where $h_{F^cV}(d_c-1)$ is sufficiently large:

\begin{lemma}\label{lemFCBig} Suppose $D<2d_c$ and $h_{F^cV}(d_c-1)\geq D-c+1$, 
then we have
\[ \sum_{k=0}^{D+d_c-c-3} h_V(k)> \sum_{i\leq j}(d_i-1)(d_j-1).\]
\end{lemma}
\begin{proof}
From Corollary~\ref{corMacDown} it follows (using $h_{F^cV}(d_c-1)\geq D-c+1\geq d_c$ and $D \leq 2d_c-1$) that
\[h_{F^cV}(k)\geq \min(k+ D-d_c+2-c,2k+1).\] The latter is at least $\sum_{i=1}^c \max(k+1-d_c+d_i,0)=h_{V_L}(k)$ for $0\leq k \leq d_c-2$. Hence $h_V(k)\geq h_{V_L}(k)$ for all $k$ and there is at least one $k$ for which the inequality is strict, which implies 
\[ \sum_{k=0}^{D+d_c-c-3} h_V(k)> \sum_{i\leq j}(d_i-1)(d_j-1).\]
\end{proof}

Corollary~\ref{corMacDown} then directly implies
\begin{lemma} \label{lemBnd1} Suppose that $d_c\geq 3$, $D<2d_c$ and $h_{F^cV}(d_c-1)\leq D-c$ holds. 
then 
\[ h_{F^cV}(d_c-2)\geq h_{F^cV}(d_c-1)-1.\]
\end{lemma}

Note that $D<d_c+c+1$ implies $c=2,d_1=2$.
\begin{lemma}\label{lemBnd2} Suppose that $d_c+c+1\leq D \leq 2d_c$.
Assume that  $X_H$ does not contain a line and that $h_{F^cV}(d_c-1)\leq D-c-2$ holds  then we have
\[ h_{F^cV}(d_c-2)\leq h_{F^cV}(d_c-1).\]
\end{lemma}
\begin{proof}
From Gorenstein duality it follows that $h_{F^cV}(D-c-2)=h_{F^cV}(d_c-1)$. Using Theorems~\ref{thmMac} and~\ref{thmGotz} it follows that $h_{F^cV}(D-c-1)\leq h_{F^cV}(D-c-2)+1$ and that if equality holds then the base locus of $F^cV_{D-c-1}$ contains a line $\ell$. Suppose first that equality holds. Since $g_1,\dots,g_c$ are contained in $ F^cV$ and $D-c-1\geq d_c$ holds, we have  that $g_1,\dots,g_c\in I(\ell)$. In particular, $X_H$ contains a line, contradicting our assumptions. Hence $h_{F^cV}(D-c-1)\leq h_{F^cV}(D-c-2)$ and by Gorenstein duality we get  $h_{F^cV}(d_c-2)\leq h_{F^cV}(d_c-1)$. 
\end{proof}

\begin{lemma}\label{lemScheme} Suppose $l:=h_{F^cV}(d_c-2)=h_{F^cV}(d_c-1)$ and $l\leq D-c-2$  hold. Then there exists a zero-dimensional scheme $\Delta$ such that $F^cV_k=I(\Delta)_k$ for $k\leq D-c-1$ and 
\[ V_k \subset \oplus_{i=1}^{c} I(\Delta)_{k+d_i-d_c}\]
for $k\leq d_c-2$.
\end{lemma}

\begin{proof}
Note that $d_c-1\leq h_{F^cV}(d_c-2)\leq D-c-2$ implies $D-c-1\geq d_c$.
Using Gorenstein duality for $F^cV$ we obtain that \begin{eqnarray*}h_{F^cV}(D-c-1)&=&h_{F^cV}(d_c-2)=h_{F^cV}(d_c-1)\\&=&h_{F^2V}(D-c-2)\\&=&l \leq D-c-2.\end{eqnarray*} In particular, by Theorem~\ref{thmGotz} we have that there is a zero-dimensional scheme $\Delta$ of length $l$ such that
$F^cV_{D-c-1}=I(\Delta)_{D-c-1}$ holds. By the construction of $F^cV$ we have then $F^cV_{k}=I(\Delta)_{k}$ for $k\leq D-c-1$.

From $d_c\leq D-c-1$ and $g_1,\dots,g_c\in F^cV$ it follows that $\Delta$ is contained in $X_H$. In particular, for any $p\in \Delta$ we have $\dim V_{d_c-1}(p)=c$. 

Let $p\in \Delta$, let $v_1\in V$ be an element of minimal degree $k$ such that $v_1(p)$ is nonzero. Since $\dim V_{d_c-1}(p)=c$ there exist $c-1$ elements $v_2,\dots, v_c\in V_{d_c-1}$ such that the determinant of $v_1,\dots,v_c$ does not vanish at $p$.  This determinant is contained in $F^cV$ and its degree equals
$ k+D-d_c-c+1$. Since $p$ is a base point of $F^cV_{D-c-1}$ it follows that $k\geq d_c-1$. Hence for $k\leq d_c-2$ we have that any element of $V_k$ vanishes along $\Delta$ and hence $V_k \subset \oplus_{i=1}^{c} I(\Delta)_{k+d_i-d_c}.$
\end{proof}

\begin{lemma}\label{lemdone2} Suppose $c=2$ and $d_1=2,d_2\geq 3$ then $X$ has at least 
\[ d_2^2-d_2+1\]
nodes and if equality holds then $X_H$ contains  a line.
\end{lemma}

\begin{proof} Since $D-2=d_2$ holds, we have that $F_{d_2}$ is base point free. The socle degree of $F^2V$ equals $D+d_c-c-3=2d_2-3$ and therefore we have $h_{F^2V}(k)\geq 2d_2-2-k$ for $d_2-1\leq k \leq 2d_2-3$. Using Gorenstein duality we get $h_{F^2V}(k)\geq k+1$ for $k\leq d_2-2$. Hence we are done if we can show $h_{V}(k)>k+1$ for some $1\leq k \leq d_2-2$.  Suppose that this is not the case then $h_{F^2V}(k)=k+1$ for $k\leq d_2-2$. By Lemma~\ref{lemFcsmall} we have that $X_H$ contains a line and therefore $X$ has at least $(d_2-1)^2+(d_2-1)+1$ nodes.
\end{proof}

\begin{proposition}\label{prpFCSmall} Suppose that $h_{F^cV}(d_c-1)\leq D-c-2$, $D<2d_c$, $d_c\geq 4$ then $X$ has at least
\[ \sum_{i\leq j} (d_i-1)(d_j-1)\]
nodes. Moreover, if equality holds then $X_H$ contains a line.
\end{proposition}

\begin{proof} If $X_H$ contains a line and the same holds for any equisingular deformation of $X$ then the result follows from Lemma~\ref{lemLine}. Assume now that there is an  equisingular deformation of $X$ such that a general hyperplane section does not contain a line. 
Since a small equisingular deformation leaves both the number of nodes and the defect invariant, we may replace $X$ by such an equisingular deformation and hence we may assume that $X_H$ does not contain a line.

If $D<d_c+c+1$ then $c=2$ and $d_1=2$. The case $d_2=2$ is covered by Example~\ref{example22}, the case $d_2>2$ is covered by Lemma~\ref{lemdone2}.

If $D\geq d_c+c+1$ then  it follows from Lemma~\ref{lemFcsmall}, \ref{lemBnd1} and~\ref{lemBnd2} that $h_{F^cV}(d_c-2)=h_{F^cV}(d_c-1)$ holds. Lemma~\ref{lemScheme} yields $h_V(k)\geq \sum \max(0,k+1-d_i+d_c)$ for $k\leq d_c-2$. Hence $h_{V}(k)\geq h_{V_L}(k)$ for $k\neq  d_c-1$. It remains to deal with $k=d_c-1$.

Suppose that $h_{F^cV}(d_c-3)<h_{F^cV}(d_c-2)$ holds then from $h_{F^cV}(d_c-2)\leq D-c-2 \leq 2d_c-c-3< 2d_c-4$ it follows that $h_{F^cV}(d_c-3)\geq h_{F^cV}(d_c-2)-1$. If equality holds then from Lemma~\ref{lemFcsmall} it follows that $X_H$ contains a line, which we excluded. Hence $h_{F^cV}(d_c-3) \geq h_{F^cV}(d_c-2)\geq d_c-1$. Corollary~\ref{corMacDown} yields  that $h_{F^cV}(k)\geq k+2$ for $1\leq k \leq d_c-3$.

In particular, it follows that \[\sum_{k=0}^{d_c-3} h_V(k)\geq d_c-3+\sum_{k=0}^{d_c-3}\sum_{i=1}^c  \max(0,k+1-d_i+d_c).\]  Using Gorenstein duality for $F^cV$ and the inequality $h_V(k)\geq h_{F^cV}(k)$ we obtain $h_V(D+d_c+c-3-k)\geq h_{F^cV}(D+d_c+c-3-k) \geq 2+k$ for $1\leq k \leq d_c-4$.
Hence $\sum_{k\neq d_c-1} h_V(k)-h_{V_L}(k) \geq 2d_c-6$.

Since $h_{F^cV}(d_c-1)\geq d_c-1$ it follows that 
\[ \sum h_V(k)-h_{V_L}(k) \geq 3d_c-D+c-6\]
Using that  $D< 2d_c$ , $c\geq 2$ and $d_c\geq 4$ we obtain that  the right hand side is at least $1$ and we are done.
\end{proof}

\begin{lemma}\label{lemFCmed}  Suppose that $X_H$ does not contain a line, that $d_c\geq 4$, $h_{F^cV}(d_c-1)\in \{D-c-1, D-c\}$  and $D<2d_c$ then we have
\[ \sum_{k=0}^{D+d_c-c-3} h_V(k)> \sum_{i\leq j} (d_i-1)(d_j-1).\]
\end{lemma}

\begin{proof}
Note that $D-c\leq 2d_c-3$ holds. From Corollary~\ref{corMacDown} it follows that $h_{F^cV}(d_c-2)\geq h_{F^cV}(d_c-1)-1$, and if equality holds then we can apply Lemma~\ref{lemFcsmall} to conclude that $X_H$ contains a line. Hence we may assume that $h_{F^cV}(d_c-2)\geq h_{F^cV}(d_c-1)$.

Suppose first $h_{F^cV}(d_c-1)=D-c$. Then from Corollary~\ref{corMacLowDeg} it follows \begin{eqnarray*} h_{F^cV}(k)&\geq &\min (2k+1,k+D-c+2-d_c)\\&\geq &\sum_{i=1}^c \min(k+d_i-d_c+1,0)=h_{V_L}(k).\end{eqnarray*}
From $h_{F^cV}(d_c-1)=D-c>h_{V_L}(d_c-1)$ one obtains that $h_V(k)\geq h_{V_L}(k)$ holds for all $k$, and for one value of $k$ we have a strict inequality.

It remains to consider the case where $h_{F^cV}(d_c-1)=D-c-1$. Then $h_{F^cV}(d_c-3)\geq D-c-2$ holds. Note that we have $h_{V_L}(d_c-3)=D-2c$. Arguing as above we have $h_V(k)\geq h_{F^cV}(k)\geq h_{V_L}(k)$ for $k\leq d_c-3$.
However, $h_{F^cV}(d_c-2)=h_{V_L}(d_c-2)-1$ can occur.  Since $X_H$ does not contain a line and $d_c\geq 4$ it follows from Lemma~\ref{lemFcsmall} that $h_{F^cV}(1)>2$ and that $h_{F^cV}(2)>3$. By Gorenstein duality we have that $h_V(D+d_c-c-4)>h_{V_L}(D+d_c-c-4)=2$ and  $h_V(D+d_c-c-5)>h_{V_L}(D+d_c-c-5)=3$. From this the claim follows.
\end{proof}

\begin{theorem}\label{thmCI} Suppose $D<2d_c$, that $X$ has defect  and that $X$ has no induced defect. Then $X$ has at least 
\[ \sum_{i\leq j} (d_i-1)(d_j-1)\]
nodes, and when equality holds then a general hyperplane section of $X$ contains a line.
\end{theorem}

\begin{proof}
If $d_1=2$ and $c=2$ then   this follows from Lemma~\ref{lemdone2}. Otherwise we have $d_c\geq 4$.
If $X_H$ contains a line then the theorem follows from Lemma~\ref{lemLine}. If $h_{F^cV}(d_c-1)\geq D-c+1$ then this follows from Lemma~\ref{lemFCBig}.
If $h_{F^cV}(d_c-1)\leq D-c-2$ then the result follows from Proposition~\ref{prpFCSmall}.
In the remaining case the result follows from  Lemma~\ref{lemFCmed}.
\end{proof}

We consider next the case where $c=2$ and $d_1=d_2$.

\begin{lemma}\label{lemP1} Suppose that $c=2$ and $d_1=d_2$  hold then $h_{V}(0)=2$.
\end{lemma}
\begin{proof}
Suppose $h_{V}(0)=1$. Then after a linear change of variables in $y_1,y_2$ (i.e., by choosing a new basis for $I(X_H)$) we may assume $(0,1)\in V$. This contradicts the fact that $0\oplus S \not \subset V$ from Proposition~\ref{prpW}.
\end{proof}

\begin{lemma} \label{lemgr} Suppose $c=2$, $d:=d_1=d_2$ and $h_{F^2V}(d-1)\geq 2d-1$ then
\[ 
\sum_{k=0}^{3d-5} h_V(k)\geq  3(d-1)^2,\]
and equality is only possible for $d=2$.
\end{lemma}
\begin{proof}
Suppose $d=2$ then $h_V(0)=2$ by Lemma~\ref{lemP1}. From Lemma~\ref{lemHighFc} it follows that $h_{F^2V}(1)=1$ and hence $h_V(0)+h_V(1)\geq 3$.

If $d=3$ then $3d-5=4$. Lemma~\ref{lemHighFc} implies  $h_{F^2V}(4)=1$. Since $h_{F^2V}(2)\geq 5$ by assumption, it follows from Corollary~\ref{corMacDown} that $h_{F^2V}(1)\geq 3$ holds. Using Gorenstein duality we get $h_{F^2V}(3)\geq 3$. Combining yields $\sum_{k=0}^4 h_V(k)\geq 13$.

Suppose now $d\geq 4$ holds. From $h_{F^2V}(d-1)\geq 2d-1$ it follows that $h_{F^2V}(k)\geq 2k+1$ holds for $0\leq k \leq d-1$ by Corollary~\ref{corMacLowDeg}. Using Gorenstein duality for $F^2V$ it follows that $h_{F^2V}(3d-5-k)\geq 2k+1$ for $0\leq k \leq d-1$.
From $h_{F^2V}(2d-4)\geq 2d-1$ one gets $h_{F^2V}(k)\geq \min (k+3,2k+1)$ for $k\leq 2d-4$. This implies that $h_{F^2V}(k)\geq k+3$ for $ (3d-5)/2\leq k \leq 2d-4$.  
An easy calculation yields
\[ 
\sum_{k=0}^{3d-5} h_V(k)> 3(d-1)^2.\]
\end{proof}

\begin{lemma}\label{lemklkl} Suppose that $X_H$ does not contain a line, $c=2$, $d=d_1=d_2\in \{3,4,5\}$ and $h_{F^2V}(d-1)\leq 2d-2$, then
\[ \sum_{k=0}^{3d-5} h_V(k)\geq 3(d-1)^2.\]
\end{lemma}
\begin{proof}

For $d=3$ we have that the socle degree of $F^2V$ is $4$. If $h_{F^2V}(1)=2$ then $X_H$ contains a line by Lemma~\ref{lemFcsmall}. Hence we have that $h_{F^2V}(1)
\geq 3$ holds. Using Gorenstein duality we obtain $h_{F^2V}(3)\geq 3$. From Corollary~\ref{corMacDown} it follows that $h_{F^2V}(2)\geq 3$. Using that $h_V(0)=2$ (Lemma~\ref{lemP1}) it follows that $\sum h_V(k)\geq 12$ holds.

For $d=4$ we have that the socle degree of $F^2V$ is $7$ and that $4\leq h_{F^2V}(3)=h_{F^2V}(4)\leq 6$ holds.

If $h_{F^2V}(3)\geq 5$ then one gets $h_{F^2V}(2)\geq 4$ and $h_{F^2V}(1)\geq 3$. Then using Gorenstein duality we get that $\sum h_{F^2V}(k)\geq 26$. Since $h_V(0)=2$ it follows that $\sum h_V(k)\geq 27$.

If $h_{F^2V}(3)=4$ and $h_{F^2V}(2)\geq 4$ holds then $h_{F^2V}(4)=4$ by duality and $h_{F^2V}(5)\leq 4$ by Theorem~\ref{thmMac}. Hence $h_{F^2V}(2)= h_{F^2V}(5)=4$. This means that the base locus of $F^2V_5$ consists of four points and we obtain that $V_k(p)=0$ for $k\leq 2$ and $p$ a base point. In particular $h_{V}(k)\geq 2 h_{F^2V}(k)$ for $k\leq 2$. This implies  $\sum h_V(k)\geq 32$.
It remains to consider the case $h_{F^2V}(3)=4$ and $h_{F^2V}(2)=3$. In this case we have by Lemma~\ref{lemFcsmall} that $X_H$ contains a line and we are done.

Suppose now $d=5$. then the socle degree is 10. Note that $h_{F^2V}(5)\geq 6$. If $h_{F^2V}(4)=5$ then by Theorem~\ref{thmGotz} we know that the base locus of $F^2V_5$ contains a line and therefore that $X_H$ contains a line. Hence $h_{F^2V}(6)=h_{F^2V}(4)\geq 6$. Using Corollary~\ref{corMacDown} it follows that $h_{F^2V}(3)\geq 5$ and $h_{F^2V}(2)\geq 4$ and that $h_{F^2V}(1)\geq 3$. Using Gorenstein duality we get that $\sum h_{F^2V}(k)\geq 44$. If we have equality in degree 3 then we have it also in degree 1 and 2 and the base locus of $F^2V_4$ consists of a line and a point and the base locus of $F^2V_6$ consists of at least 5 collinear points. This implies that $V_k(p)=0$ for $k\leq 2$ and all four base points. In particular we get a contribution of at least 8 from $P^1V$, and hence $
\sum_k h_V(k)$  is at least 52.
If $h_{F^2V}(3)\geq 6$ then $h_{F^2V}(2)\geq 4$. In this case we obtain $\sum h_{F^2V}(k)\geq 48$ and we are done.

Now assume that $h_{F^2V}(4)\geq 6$ then $h_{F^2V}(2)\geq 5$. This implies that the total contribution of $F^2V$ is at least 48. Since there is a contribution from $P^1V$ we are done.
\end{proof}

\begin{lemma}\label{lemklgr} Suppose that $X_H$ does not contain a line and that $d:=d_1=d_2>5$ and $h_{F^2V}(d-1)\leq  2d-2$ hold then
\[ \sum_{k=0}^{3d-5} h_V(k)\geq 3(d-1)^2.\]
\end{lemma}

\begin{proof} Suppose first that $h_{P^1V}(d-1)\geq 1$. We show first that $h_{V}(d-1)\geq 2d-3$.
If $h_{F^2V}(d)>h_{F^2V}(d-1)$ then by Theorem~\ref{thmGotz} we obtain that the base locus of $F^cV_d$ contains a line and therefore that $X_H$ contains a line. Hence $h_{F^2V}(d)\leq h_{F^2V}(d-1)$ holds. Using Gorenstein duality we get by the same argument  $h_{F^2V}(d)=h_{F^2V}(2d-5)\geq h_{F^2V}(2d-4)= h_{F^2V}(d-1)$ and hence $h_{F^2V}(d)=h_{F^2V}(d-1)$. 
Using Gorenstein duality again we get that $h_{F^2V}(2d-3)=h_{F^2V}(d-2)\geq d-1$ and therefore that $h_{F^2V}(d-1)\geq d-1$. To conclude this case if $h_{P^1V}(d-1)\geq d-1$ then we obtain $h_{V}(d-1)\geq 2d-2$. If $h_{P^1V}(d-1)<d-1$ then we get that $h_{P^1V}(d)<h_{P^1V}(d-1)$ and hence that  $h_V(d-1)>h_V(d)=2d-4$ holds. In both cases we have that $h_V(d-1)\geq 2d-3$.

We consider next degree $d-2$.
If $h_{F^2V}(d-2)=d-1$ and either $h_{F^2V}(d-2)<h_{F^2V}(d-1)$ or $h_{F^2V}(d-3)<h_{F^2V}(d-2)$ holds then we have that $F^2V_{d-1}$ is the degree $d-1$ part of the ideal of a line and hence that $X_H$ contains a line by Lemma~\ref{lemFcsmall}.
Hence we may exclude this and we have in any case that $h_{F^2V}(k)\geq k+2$ for $k\leq d-3$ and $h_{F^2V}(d-2)\geq d-1$.
If $h_{P^1V}(d-1)\geq d-1$ then we have $h_{P^1V}(k)\geq k+1$ for $k\leq d-2$ by Corollary~\ref{corMacLowDeg}, hence $h_V(k)\geq 2k+2$ for $k\leq 2$ and we are done.
We assume now that $h_{P^1V}(d-1)<d-1$.

If $h_{F^2V}(d-2)<h_{F^2V}(d-1)$ then $h_{F^2V}(d-2)\geq d$ and therefore $h_{F^2V}(d-1)\geq d+1$. Set $a=h_{F^2V}(d-1)-d+1$ and $b=h_{P^1V}(d-1)$. Then $2\leq a \leq d-1$; $1\leq b \leq d-1$ and $a+b\geq d-2$.

From Theorem~\ref{thmMac} it follows that $h_{F^2V}(k)\geq \min(k+a,2k+1)$ and $h_{P^1V}(k)\geq \min(k+1,b)$. If $h_V(k)< 2k+2$ then $k+a+b<2k+2$, i.e. $k\in \{d-2,d-3\}$. The total difference $\sum_{k=0}^{d-1}h_{V_L}(k)-h_V(k)$ is at most 3. By duality we have that $\sum_{k=d}^{3d-5} h_V(k)-h_{V_L}(k)$ at least $d-1$ and since  $d> 5$  we are done.

The final case is  $h_{P^1V}(d-1)=0$. In this case $h_{F^2V}(d)=h_{V}(d)\geq 2d-4$ and by Theorem~\ref{thmMac} we have that $h_{F^2V}(d)\leq 2d-1$. Let $a=h_{F^2V}(d)-d\geq d-4$. Then $h_{F^2V}(k)\geq \min (k+a,2k+1)$. The total miss $\sum_{k=0}^{d-1} h_{V_L}(k)-h_{V}(k)$ is at most
$d+6$. If $d\geq 6$ then $a\geq 2$ and by duality we have that $\sum_{k\geq d} h_V(k)-h_{V_L}(k)$ is at least $2d+1$, which is again at least $d+6$. 
\end{proof}

\begin{theorem}\label{thmCIeq} Suppose $X$ is a nodal complete intersection of bidegree $(d,d)$ with defect and without induced defect. Then $X$ has at least $3(d-1)^2$ nodes.
\end{theorem}
\begin{proof}
If $d=2$ then by Lemma~\ref{lemP1} we have that $h_{V}(0)=2$. Since $X$ has defect we have that $h_V(1)\geq 1$ and we are done.

Suppose now that $d>2$ holds.

If $X_H$ contains a line and the same holds for any equisingular deformation then the result follows from Lemma~\ref{lemLine}. If $X_H$ contains a line but this property does not hold for any equisingular deformation then we may replace $X$ by this equisingular deformation and therefore assume that $X_H$ does not contain a line.
Depending on $d$ and $h_{F^cV}(d-1)$ this is covered by one of  Lemma~\ref{lemgr}, Lemma~\ref{lemklkl} or Lemma~\ref{lemklgr}.
\end{proof}

\bibliographystyle{plain}
\bibliography{remke2}

\end{document}